\documentclass[11pt]{amsart}
\usepackage{amssymb,amsmath, amsthm, mathrsfs, amscd, amssymb, enumitem}
\usepackage{thmtools}
\usepackage{hyperref}
\usepackage{cleveref,comment,xcolor}
\usepackage[margin=34mm]{geometry}

\hypersetup{hypertexnames=false}

\newtheorem{theorem}{Theorem}[section]
\newtheorem{lemma}[theorem]{Lemma}
\newtheorem{prop}[theorem]{Proposition}

\theoremstyle{definition}

\theoremstyle{remark}
\newtheorem{remark}[theorem]{Remark}

\numberwithin{equation}{section}

\hypersetup{
	colorlinks,
	linkcolor={red!25!black},
	citecolor={blue!25!black},
	urlcolor={blue!25!black}
}

\renewcommand{\hat}{\widehat}

\renewcommand{\S}{\mathbb{S}}

\newcommand{\R}{\mathbb{R}}

\newcommand{\Q}{\mathcal{Q}}
\newcommand{\bB}{\mathbb{B}}
\newcommand{\Z}{\mathbb{Z}}

\newcommand{\W}{\mathbb{W}}
\newcommand{\T}{\mathbb{T}}
\newcommand{\sq}{\square}

\newcommand{\Bb}{\mathcal{B}_{\operatorname{broad}}}
\newcommand{\Bn}{\mathcal{B}_{\operatorname{narrow}}}
\newcommand{\Yb}{Y_{\operatorname{broad}}}
\newcommand{\Yn}{Y_{\operatorname{narrow}}}

\newcommand{\les}{\lesssim}
\newcommand{\ges}{\gtrsim}
\newcommand{\less}{\lessapprox}
\newcommand{\eps}{\epsilon}
\newcommand{\la}{\lambda}

\begin{document}
	\title[Weighted refined decoupling with concentrated frequencies]{Weighted  decoupling with lower-dimensional frequency localization}	
	\author{Jongchon Kim}
	\address{Department of Mathematics, City University of Hong Kong, Hong Kong SAR, China}
	\email{jongckim@cityu.edu.hk}
	\thanks{Supported in part by a grant from the Research Grants Council of the Hong Kong Administrative Region, China (Project No. CityU 11308924).}
	
	\begin{abstract} 
		We prove weighted $L^2$ and refined $L^p$ decoupling estimates for functions whose Fourier transforms are supported in a small neighborhood of the unit sphere or the truncated paraboloid  with an additional lower-dimensional frequency localization property. 	As a special case, we recover the fractal $L^2$ restriction estimate of Du and Zhang, with a sharper dependence on the density of the weight. We also derive weighted refined decoupling estimates related to the Falconer distance set problem, improving earlier results under the stronger assumption that the underlying weight is $\alpha$-dimensional at every scale.
	\end{abstract}
	
	\maketitle

\section{Introduction}
Let $S\subset\R^d$ be a compact strictly convex $C^2$ hypersurface with Gaussian curvature comparable to $1$, such as the unit sphere or the truncated paraboloid. 
We prove weighted decoupling estimates for functions whose Fourier transforms are supported in the $R^{-1}$-neighborhood $N_{R^{-1}}(S)$ of $S$, under an additional frequency localization assumption. This type of estimate was introduced by Du, Ou, Ren, and Zhang \cite{du2023weighted} in connection with the Falconer distance set problem. We obtain an improved bound under the stronger assumption that the underlying weight is $\alpha$-dimensional at every scale. Weighted variants of decoupling estimates have also appeared recently in other contexts; see, for example, \cite{GanWu, KimWeightedDec, CLPY}.

We first recall the wave packet setup. Let $\{ \theta\}$ denote a canonical covering of $N_{R^{-1}}(S)$ by parallelepipeds of dimensions $R^{-1/2} \times \cdots \times R^{-1/2} \times R^{-1}$. For each $\theta$, let $\T_\theta$ denote a collection of parallelepipeds dual to $\theta$ forming a partition of $\R^d$, each of dimensions $R^{1/2} \times \cdots \times R^{1/2} \times R$. Let $v(T)$ denote the long direction of $T\in \T_\theta$ orthogonal to $\theta$. A wave packet $f_T$, $T\in \T_\theta$, has Fourier support in a small dilate of $\theta$ and decays rapidly away from $T$. Any function with Fourier support in $N_{R^{-1}}(S)$ admits a decomposition into wave packets associated with $\T(R) = \cup_\theta \T_\theta$; see e.g. \cite[Appendix A.2]{decouplingstudy}.

Let $\W \subset \T(R)$. Following \cite{du2023weighted}, we say that a sum of wave packets $\sum_{T\in \W} f_T$  has \emph{$(r,m)$-concentrated frequencies}, for some $r\geq R^{-1/2}$ and $1\leq m\leq d$, if there exists an $m$-dimensional subspace $V\subset \R^d$ such that 
\[ \angle(v(T), V) \leq r,\;\; \text{for all } T \in \W.\]
This condition is vacuous when $m=d$ or  $r\geq \pi/2$.

For $\alpha \in [0,d]$ and a set $Y\subset \R^d$, define
\begin{align*}
	\gamma_{\alpha,R}(Y) =  \sup_{1\leq \rho \leq R} \sup_{B_\rho} \rho^{-\alpha} |Y\cap B_\rho|,
\end{align*}
where the second supremum is taken over all cubes $B_\rho\subset\R^d$ of side
length $\rho$, referred to as a $\rho$-cube. We also define
\[
 \overline{\gamma_{\alpha,R}}(Y) =  \sup_{R^{1/2} \leq \rho \leq R} \sup_{B_\rho} \rho^{-\alpha} |Y\cap B_\rho|,
\]  
which ignores scales below $R^{1/2}$. Additionally, we define 
\[ \mathcal{A}_{d,m,\alpha,\beta}(R, Y) =   \gamma_{\alpha,R^{1/2}}(Y) \overline{\gamma_{\beta,R}}(Y) R^{\beta}  
\sup_{B_{R^{1/2}}} |Y\cap B_{R^{1/2}}|^{m-1}	R^{-\frac{m(d+1)}{2}}. \]
This quantity records the density parameters of $Y$ that appear in the main estimates.
We denote by $A\les B$ an inequality $A\leq CB$ with an absolute constant $C$, and by $A\less B$ the estimate $A\leq C_\eps R^\eps B$ for every $\eps>0$. We can now state our main result.

\begin{theorem}\label{thm:LpWeightedDecouplingConcentrated}
	Let $2\leq m \leq d$. Let $f = \sum_{T \in \W} f_T$ be a sum of wave packets with $(R^{-1/2}, m)$-concentrated frequencies, and let $Y\subset B_R$ be a union of unit cubes.
	Suppose  $0\leq \alpha, \beta \leq d$ satisfy $\beta \geq \alpha +m - d$. 	Then
\[
		\| f \|_{L^2(Y)} \less \mathcal{A}_{d,m,\alpha,\beta}(R,Y)^{\frac{1}{m(m+1)}} 
		\| f \|_{L^2}.
\]
\end{theorem}
\begin{remark}
	The same estimate holds more generally for any union of unit cubes in $\R^d$. Indeed, the general case follows by partitioning	$\R^d$ into $R$-cubes and summing the local estimates in \Cref{thm:LpWeightedDecouplingConcentrated}.
\end{remark}
\begin{remark}
	The bound is monotonic in $\beta$ since 
	\[ \beta \leq \beta'  \;\; \implies \;\; \overline{\gamma_{\beta,R}}(Y) R^{\beta}  \leq \overline{\gamma_{\beta',R}}(Y) R^{\beta'}.   \]
 Accordingly, \Cref{thm:LpWeightedDecouplingConcentrated} gives the strongest bound when $\beta = \max(0,\alpha + m - d)$.
\end{remark}
\begin{remark}
	The factor $\gamma_{\alpha,R^{1/2}}(Y)$ allows the induction to close for $\beta<m$ in the proof of \Cref{thm:LpWeightedDecouplingConcentrated}. It is
	not clear, however, whether this factor is genuinely necessary.
On the other hand, when $\beta \geq m$, this factor can be  removed by taking $\alpha = d$, since $\gamma_{d,R^{1/2}}(Y) =1$ whenever $Y$ contains at least one unit cube.
\end{remark}

 \Cref{thm:LpWeightedDecouplingConcentrated} yields corresponding weighted $L^2$ estimates for the Fourier extension operator associated with the surface $S$. 
 To keep the presentation simple, we focus on the case $m=d$, where $S = \{ (\xi,|\xi|^2) \in \R^{n+1}, \;\; \xi \in B_1^n(0) \}$ is the truncated paraboloid. Let 
\[ e^{it\Delta} f(x) = (2\pi)^{-n} \int_{\R^n} e^{ix\cdot \xi} e^{i t |\xi|^2} \hat{f}(\xi) d\xi \]
denote the solution to the Schr\"odinger equation $iu_t = \Delta u$ with initial data $f$. For $f$ with Fourier transform supported in $B_1^n(0)$, we let 
\begin{equation}\label{eqn:U}
	 Uf(x,t) = \varphi(R^{-1}t) e^{it\Delta} f(x), \;\; (x,t)\in \R^{n+1},
\end{equation}
where $\hat{\varphi}$ is a smooth bump function such that $1_{[-1/2,1/2]} \leq \hat{\varphi}\leq 1_{[-1,1]}$. Note that $\widehat{Uf}$ is supported on the $R^{-1}$-neighborhood of the truncated paraboloid in $\R^{n+1}$. Moreover, $\| e^{it\Delta} f\|_{L^2(Y)} \sim \| Uf \|_{L^2(Y)}$ for $Y\subset B_R^{n+1}(0)$ and, by Plancherel's theorem, $\| Uf \|_{L^2}^2 \sim R \| f\|_{L^2}^2$. Thus, applying  \Cref{thm:LpWeightedDecouplingConcentrated} to $Uf$ in the case $m=d=n+1$ yields the following. 
\begin{theorem}\label{cor:fractalRestriction}
	Let $n\geq 1$, $0\leq \alpha \leq n+1$ and $Y\subset B^{n+1}_{R}(0)$ be a union of unit cubes. For all $f$ whose Fourier transform is supported in $B_1^n(0)$, we have
\[
		\| e^{it\Delta} f\|_{L^2(Y)}\less \Big( \gamma_{\alpha,R^{1/2}}(Y) \overline{\gamma_{\alpha,R}}(Y) R^{\alpha}  
		\sup_{B_{R^{1/2}}} |Y\cap B_{R^{1/2}}|^{n}	 \Big)^{\frac{1}{(n+1)(n+2)}} \| f \|_{L^2}. 
\]
\end{theorem}

 \Cref{cor:fractalRestriction} sharpens the fractal $L^2$ restriction estimate of Du and Zhang \cite[Theorem 1.6]{DZ}; the product $\gamma_{\alpha,R^{1/2}}(Y) \overline{\gamma_{\alpha,R}}(Y)$ replaces the factor $\gamma_{\alpha,R}(Y)^2$ in that paper. The improvement is significant when one of the two factors is much smaller than $\gamma_{\alpha,R}(Y)$.  We give such an example for which \Cref{cor:fractalRestriction} is sharp. 
\begin{prop}\label{prop:sharpnessFractal}
	For each $\frac{n}{2}\leq \alpha < n+1$, there exist   $f$ with $\hat{f}$ supported in $B_1^n(0)$  and a union of unit cubes $Y\subset B^{n+1}_{R}(0)$ such that $\gamma_{\alpha,R^{1/2}}(Y) \sim 1$ and $\overline{\gamma_{\alpha,R}}(Y) = \gamma_{\alpha,R}(Y) \sim R^{\frac{n+1-\alpha}{2}}$ for which 
	\[ 	\| e^{it\Delta} f\|_{L^2(Y)} \gtrapprox \Big( \gamma_{\alpha,R^{1/2}}(Y) \overline{\gamma_{\alpha,R}}(Y) R^{\alpha}  
	\sup_{B_{R^{1/2}}} |Y\cap B_{R^{1/2}}|^{n}	 \Big)^{\frac{1}{(n+1)(n+2)}} \| f \|_{L^2}. \]
\end{prop}	

We recall that the weighted $L^2$ restriction estimate in \cite{DZ} yields sharp $L^2$ bounds for the Schr\"odinger maximal function and has further applications to, e.g., the spherical average Fourier decay rates of fractal measures and the Falconer distance set problem (cf.  \cite{Carleson, falconer_1985, Mattila_Sph}). We also mention related work  \cite{Shayya2} on weighted Fourier restriction estimates for weights of low fractal dimension, as well as \cite{du2024Lpweighted}, which studies weighted $L^2\to L^p$ restriction estimates under lower-dimensional wave packet concentration properties.

Next, we present a weighted refined decoupling estimate which follows from  \Cref{thm:LpWeightedDecouplingConcentrated}. In what follows, we let $\W \subset \T(R)$ and work under the assumptions that 
\begin{itemize}
	\item $\{ Q\}$ is a collection of $R^{1/2}$-cubes in an $R$-cube $B_R$, such that each $Q$ is contained in at most $M\geq 1$ of the dilated tubes $\{ 2T\}_{T\in \W}$.\footnote{In \cite{GIOW,du2023weighted}, refined decoupling is stated with $M$ counting
		the number of tubes $T\in\W$ that intersect a given cube $Q$. For our purposes,
		we use the present formulation, which can be obtained from the proof in
		\cite{GIOW,du2023weighted}.}  Here  $2T$ denotes the concentric $R^{\eps_1}$ dilate of $T$ for  $\eps_1>0$ sufficiently small relative to  $\eps>0$.
	\item $Y\subset \cup Q$ is a union of unit cubes. 
\end{itemize}

By interpolating  \Cref{thm:LpWeightedDecouplingConcentrated} and  refined decoupling \cite{GIOW, du2023weighted} (see \Cref{thm:refined_decouplingDuGeneral} below for a statement), we obtain the following.
\begin{theorem}\label{thm:LpWeightedDecouplingConcentratedRefined}
	Let $2\leq m \leq d$. Let $f = \sum_{T \in \W} f_T$ be a sum of wave packets with $(R^{-1/2}, m)$-concentrated frequencies. Suppose  $0\leq \alpha, \beta \leq d$ satisfy $\beta \geq \alpha +m - d$. Then for $2\leq p \leq p_m:= \frac{2(m+1)}{m-1}$, 
	\begin{align*}
		\| f \|_{L^p(Y)} \less \mathcal{A}_{d,m,\alpha,\beta}(R, Y)^{\frac{1}{m}\big(\frac{1}{p}-\frac{1}{p_m}\big)} 
		M^{\frac{1}{2}-\frac{1}{p}}   \bigg( \sum_{T\in \W} \| f_T \|_{L^p}^p \bigg)^{\frac{1}{p}}.
	\end{align*}
\end{theorem}
When $\beta \leq m$ and $d-\frac{m+1}{2}\leq \alpha \leq d$, \Cref{thm:LpWeightedDecouplingConcentratedRefined} yields sharp estimates for variants of the example in Proposition \ref{prop:sharpnessFractal}; see \Cref{sec:sharpness}.

We compare \Cref{thm:LpWeightedDecouplingConcentratedRefined} with the earlier result \cite[Theorem 1.2(b)]{du2023weighted},
\[ 
	\| f \|_{L^p(Y)} \less R^{-\frac{1}{2} (d-\alpha)(\frac{1}{p}-\frac{1}{p_m})} 
	M^{\frac{1}{2}-\frac{1}{p}}   \bigg( \sum_{T\in \W} \| f_T \|_{L^p}^p \bigg)^{\frac{1}{p}},
\]
which was proved under the assumption that $Y$ is $\alpha$-dimensional at the scale $R^{1/2}$:  \[ \sup_{B_\rho} \rho^{-\alpha} |Y\cap B_\rho| \les 1, \;\;  \text{for} \;\;  \rho = R^{1/2}.\] 
This estimate is sharp in general for $d-\frac{m+1}{2}\leq\alpha\leq d$, and remains sharp for $\alpha\geq m$ even under the stronger assumption that $Y$ is $\alpha$-dimensional \emph{up to scale} $R$, that is, $ \gamma_{\alpha,R}(Y) \les 1$.   Under this stronger assumption, \Cref{thm:LpWeightedDecouplingConcentratedRefined} yields 
\[ 
			\| f \|_{L^p(Y)} \less R^{-\frac{1}{2} ( d+1 -  \frac{m+1}{m} \alpha)(\frac{1}{p}-\frac{1}{p_m})} 
			M^{\frac{1}{2}-\frac{1}{p}}   \bigg( \sum_{T\in \W} \| f_T \|_{L^p}^p \bigg)^{\frac{1}{p}},
\]
which improves on  the preceding estimate when $\alpha < m$. 
	
\subsection*{Falconer distance set problem}
We discuss  implications of our weighted refined decoupling estimates for the Falconer distance set problem. Let $d\geq 2$ and $E\subset \R^d$ be a compact set. Define the distance set of $E$ by
\[\Delta(E) = \{ |x-y|: x,y\in E\}. \]
Falconer \cite{falconer_1985} proved that the distance set has positive Lebesgue measure whenever $\dim_H(E)>\frac{d+1}{2}$. It is conjectured that $|\Delta(E)|>0$ under the weaker condition $\dim_H(E)>\frac{d}{2}$. Despite numerous partial results and significant advances, this conjecture remains open; see \cite{GIOW, du2023Falconer} and references therein. Currently, the best-known thresholds are $\dim_H(E)>\frac{5}{4}$ for $d=2$ \cite{GIOW} and  
\[ \dim_H(E)>\frac{d}{2}+\frac{d}{4d+2}\] 
for $d\geq 3$ \cite{du2023Falconer}. In fact, these works prove a stronger result that there is a point $x \in E$ for which the pinned distance set 
\[ \Delta_x(E) = \{ |x-y| : y \in E \}\] has positive Lebesgue measure. 

A key ingredient in these works is a decomposition of a fractal measure into good and bad parts. In \cite{du2023weighted}, the authors introduced a variant of this decomposition that simplifies the treatment of the bad part, at the cost of making the good part more difficult to handle, compared with \cite{du2023Falconer}. To deal with the good part,  they established weighted refined decoupling estimates for functions with $(r,m)$-concentrated frequencies, recovering the results of \cite{du2023Falconer} for $d\geq 4$. On the other hand, in the case $d=3$, this method yielded a weaker threshold of $\dim_H(E)> 1.757\cdots$, compared to $\dim_H(E) > \frac{12}{7}=1.714\cdots$ obtained in \cite{du2023Falconer}.

It remains open whether the threshold 
$\dim_H(E) > \frac{12}{7}$ can be recovered  in $\R^3$ within the framework of \cite{du2023weighted} by further improving the decoupling estimates for functions with $(r,m)$-concentrated frequencies. Our weighted refined decoupling estimates provide partial progress in this direction.
\begin{theorem}\label{thm:Falconer}
	Let $E\subset \R^3$ be a compact set. Then there is a point $x\in E$ such that $|\Delta_x(E)|>0$ whenever $\dim_H(E)>
	\frac{\sqrt{241}-5}{6} = 1.754\cdots$.
\end{theorem}
\Cref{thm:Falconer} gives a modest improvement on the threshold from \cite{du2023weighted}. However, it falls short of recovering the threshold obtained in \cite{du2023Falconer}.

We close the introduction by briefly discussing the proof of the main result,
\Cref{thm:LpWeightedDecouplingConcentrated}.
The argument is based on a weighted version of the standard broad--narrow analysis (see, e.g., \cite{BG, Guth2, DZ}), following \cite{KimWeightedDec}. In the broad case, we combine refined decoupling \cite{GIOW, du2023weighted} with the multilinear Kakeya estimates \cite{BCT} to obtain a multilinear estimate. This approach goes back to \cite{DGL, DGLZ} in the setting of multilinear refined Strichartz estimates. In the narrow case, we use a weighted $L^2$ estimate derived from refined decoupling at an intermediate scale, and then apply induction on scales, adapting ideas from \cite{DZ, du2023weighted}. 

Throughout the proof, we work primarily with $L^2$ estimates, which slightly simplifies the proof in both the broad and the narrow cases. In particular, our formulation of the multilinear weighted $L^2$ estimate, inspired by \cite{DemRefined}, allows us to avoid the technical argument in \cite{DZ} that effectively replaces a product of integrals by an integral of a product. 
 
This paper is organized as follows. In \Cref{sec:sharpness}, we discuss the sharpness of the weighted refined decoupling estimate, \Cref{thm:LpWeightedDecouplingConcentratedRefined},  and prove Proposition \ref{prop:sharpnessFractal}. In \Cref{sec:implications}, we deduce  \Cref{thm:LpWeightedDecouplingConcentratedRefined} from \Cref{thm:LpWeightedDecouplingConcentrated} and derive several weighted refined decoupling estimates. 
In \Cref{sec:Falconer}, following \cite{du2023weighted}, we apply these estimates to the Falconer distance set problem.
 In \Cref{sec:refined}, we establish a multilinear weighted $L^2$ estimate that will be used in the proof of \Cref{thm:LpWeightedDecouplingConcentrated}. 
Finally, in \Cref{sec:proofMain}, we prove \Cref{thm:LpWeightedDecouplingConcentrated}.

Throughout the paper, we denote by  $p_d = \frac{2(d+1)}{d-1}$ the critical exponent in the Bourgain-Demeter decoupling theorem \cite{BD} associated with the surface $S\subset \R^d$. We denote by $B_R$ a cube in $\R^d$ of side length $R$, and write $B_R^d(c)$ when we wish to specify the ambient dimension and the center of the cube.

\section{Sharp examples}\label{sec:sharpness}
In this section, we prove sharpness results for  \Cref{thm:LpWeightedDecouplingConcentratedRefined} and Proposition \ref{prop:sharpnessFractal}. We use an example from \cite{du2023weighted}, which is adapted from \cite{BBCRV}; see also \cite{DKWZ} for a related adaptation of Bourgain's example \cite{Bo_Sch} for the Schr\"odinger maximal function.
\begin{prop}\label{prop:sharpRefinedD}
	Let $d-\frac{m+1}{2}\leq \alpha \leq d$ and $0\leq \beta \leq m$. Let  $\{ Q\}$ be a collection of $R^{1/2}$-cubes which partition the slab $B^{d-m}_{R^{1/2}}(0) \times B^{m}_{R}(0)$. There exist a function $f=\sum_{T\in \W} f_T$ with $(R^{-1/2},m)$-concentrated frequencies and a union of unit cubes $Y \subset B^{d-m}_{R^{1/2}}(0) \times B^{m}_{R}(0)$ satisfying the following properties. 
	\begin{enumerate}[label=(\roman*)]
		\item $f(x,t) = U \psi (x,t)$ for a function $\psi \in L^2(\R^{d-1})$, where $U$ is defined in \eqref{eqn:U}.
		\item $Y$ is $\alpha$-dimensional up to scale $R^{1/2}$, i.e., $\gamma_{\alpha,R^{1/2}}(Y) \sim 1$. 
		\item $\max_{B_{R^{1/2}}} |Y\cap B_{R^{1/2}}| \sim R^{\frac{\alpha}{2}}$, and  $\overline{\gamma_{\beta, R}}(Y)  R^\beta \sim |Y| \sim R^{\frac{\alpha+m}{2}}$.
		\item Each $R^{1/2}$-cube $Q \subset B^{d-m}_{R^{1/2}}(0) \times B^{m}_{R}(0)$ is contained in at most $M$ dilated tubes $\{ 2T\}_{T\in \W}$ for $M\approx R^\frac{d-\alpha}{p_m}$.
		\item For each $2\leq p\leq p_m$, 
		\begin{align*}
			\| f\|_{L^p(Y)}  
			\gtrapprox \mathcal{A}_{d,m,\alpha,\beta}(R, Y)^{\frac{1}{m}\big(\frac{1}{p}-\frac{1}{p_m}\big)} M^{\frac{1}{2}-\frac{1}{p}} (\sum_{T\in \W} \| f_{T} \|_{L^p}^p)^{\frac{1}{p}}. 
		\end{align*}
	\end{enumerate}
\end{prop}
\begin{proof}
We recall the example from \cite{du2023weighted}. Let 
\[ \kappa = \frac{d-\alpha}{2(m+1)}.\] 
We note that $0\leq \kappa  \leq \frac{1}{4}$ for the range of $\alpha$ under consideration.
For a small $0<c<1$, let $J = (R^{-\kappa}\Z^{m-1}) \cap B^{m-1}_{c}(0)$. For each $j\in J$, let $\hat{\psi_{j}}$ be a smooth bump function which approximates the indicator function of the rectangle
\[ O_j =B^{d-m}_{cR^{-1/2}}(0) \times B^{m-1}_{cR^{-1}}(j), \]
and take $\psi = \sum_{j\in J} \psi_j$. 
We write $x=(x',x'') \in \R^{d-m} \times \R^{m-1}$ and $\xi=(\xi',\xi'') \in \R^{d-m} \times \R^{m-1}$. 
Define $f=U\psi$, so that $f  = \sum_{j\in J} f_j$, where 
\[ f_j(x,t) = U \psi_{j}(x,t) =  \varphi(R^{-1} t) c_d \int e^{i(x',x'',t) \cdot (\xi',\xi'', |\xi'|^2+|\xi''|^2)} \hat{\psi_{j}}(\xi) d\xi.\]
Then $|f_{j}| \sim |O_j|$ on the slab $B^{d-m}_{R^{1/2}}(0) \times B^{m}_{R}(0)$ and $|f_j|$ decays rapidly away from this slab. Thus, we have
\[ \| f_j \|_{L^p} \sim | O_j| R^{\frac{d+m}{2p}}. \]

Let $f_j = \sum_{T\in \T_j} f_T$ denote a wave packet decomposition of $f_j$. Let $\W = \cup_{j\in J} \T_j$. Then each $T \in \W$ is parallel to the $m$-dimensional subspace $V=\{0\}^{d-m} \times \R^m$. Therefore, 
\[ f= \sum_{j \in J} f_{j} =\sum_{T\in \W} f_T\]
has  $(R^{-1/2},m)$-concentrated frequencies. Moreover, every $Q$ is contained in $\leq M$ of the dilated tubes $\{ 2T\}_{T\in \W}$, where 
\[ M \approx \# J \sim R^{\kappa(m-1)} = R^{ \frac{d-\alpha}{2(m+1)}(m-1)} = R^{\frac{d-\alpha}{p_m}}.  \] 

Let 
\[ Y = B^{d-m}_{R^{1/2}}(0) \times \Big(   B^m_R(0)\cap \big( (2\pi R^{\kappa}\Z^{m-1} \times 2\pi R^{2\kappa} \Z)   + B^m_1(0) \big) \Big).  \]
Choosing $c>0$ sufficiently small, we have, for $(x,t)\in Y$ and
$\xi\in\bigcup_j O_j$,
\[ (x,t)\cdot(\xi, |\xi|^2) =  (x',x'',t) \cdot (\xi',\xi'', |\xi'|^2+|\xi''|^2) \in 2\pi\Z + [-0.01,0.01]. \]
Thus, \[ |f(x,t)| \sim \# J | f_j (x,t)| \sim \# J | O_j|, \;\; (x,t) \in Y.\]  Note that 
\[ |Y| \sim R^{\frac{d-m}{2}} R^{(1-\kappa)(m-1) + 1-2\kappa} = R^{\frac{d+m}{2} - \kappa(m+1)} = R^\frac{m+\alpha}{2}.\]

We have 
\begin{align*}
	\| f\|_{L^p(Y)} &\gtrsim \# J | O_j| |Y|^{\frac{1}{p}} \sim  \# J | O_j| R^{\frac{m+\alpha}{2p}} \\
M^{\frac{1}{2}-\frac{1}{p}} (\sum_{T\in \W} \| f_{T} \|_{L^p}^p)^{\frac{1}{p}} &\les 	M^{\frac{1}{2}-\frac{1}{p}} (\sum_{j\in J} \| f_j \|_{L^p}^p)^{\frac{1}{p}}  \approx (\# J)^{\frac{1}{2}} |O_j| R^{\frac{d+m}{2p}}.
\end{align*}
Hence, we have the lower bound
\begin{equation}\label{eqn:lower}
	\frac{	\| f\|_{L^p(Y)} }{M^{\frac{1}{2}-\frac{1}{p}} (\sum_{T\in \W} \| f_{T} \|_{L^p}^p)^{\frac{1}{p}}  } \gtrapprox (\# J)^{\frac{1}{2}} R^{-\frac{d-\alpha}{2p}}  \sim  R^{-\frac{d-\alpha}{2} ( \frac{1}{p}-\frac{1}{p_m})}. 
\end{equation}

On the other hand, we have
	\[	|Y\cap B_\rho| \les \begin{cases}
		\rho^{d-m}  &  \rho \in [1,R^{\kappa}], \\
		\rho^{d-m} (\rho R^{-\kappa})^{m-1} \leq  \rho^{d-\frac{m+1}{2}} & \rho \in [R^{\kappa},R^{2\kappa}], \\
		\rho^{d-m} (\rho R^{-\kappa})^{m-1} (\rho R^{-2\kappa}) = R^{-\frac{d-\alpha}{2}} \rho^{d} & \rho \in [R^{2\kappa}, R^{1/2}].
	\end{cases}\]
	It follows from these bounds that $Y$ is $\alpha$-dimensional up to scale $R^{1/2}$. Moreover, 
\[ \gamma_{\alpha,R^{1/2}}(Y) \sim 1 \; \text{ and } \; \max_{B_{R^{1/2}}} | Y \cap B_{R^{1/2}} | \sim R^{\alpha/2}. \]

For $R^{1/2}\leq \rho \leq R$,
\begin{align*}
	|Y\cap B_\rho| \les R^{\frac{d-m}{2}} (\rho R^{-\kappa})^{m-1} (\rho R^{-2\kappa}) =R^{\frac{\alpha-m}{2}} \rho^m
\end{align*}
and there exists $B_\rho$ where the above inequality is tight.  Thus, when $\beta \leq m$, 
\[  \overline{\gamma_{\beta,R}} (Y) R^\beta \sim R^{\frac{m+\alpha}{2}} \sim |Y|. \]
Combining these estimates, we get
\[ \Big( \gamma_{\alpha,R^{1/2}}(Y) \overline{\gamma_{\beta,R}}(Y) R^{\beta}
\sup_{B_{R^{1/2}}} |Y\cap B_{R^{1/2}}|^{m-1}	R^{-\frac{m(d+1)}{2}} \Big)^{\frac{1}{m}\big(\frac{1}{p}-\frac{1}{p_m}\big)} \sim R^{-\frac{d-\alpha}{2} ( \frac{1}{p}-\frac{1}{p_m})},  \]
which, combined with \eqref{eqn:lower}, verifies the claim. 	
\end{proof}

A special case of Proposition \ref{prop:sharpRefinedD} yields Proposition \ref{prop:sharpnessFractal}.
\begin{proof}[Proof of Proposition \ref{prop:sharpnessFractal}]
	  Let $d=m=n+1$, $\alpha=\beta$ and $p=2$. Consider the example from Proposition \ref{prop:sharpRefinedD}. We have 
	  \[ \| U\psi \|_{L^2(Y)} \gtrapprox  \Big( \gamma_{\alpha,R^{1/2}}(Y) \overline{\gamma_{\alpha,R}}(Y) R^{\alpha}
	  \sup_{B_{R^{1/2}}} |Y\cap B_{R^{1/2}}|^{n}  \Big)^{\frac{1}{(n+1)(n+2)}} R^{-\frac{1}{2}}  \| U\psi \|_{L^2}.   \]
	  Since $R^{-\frac{1}{2}}  \| U\psi \|_{L^2} \sim \| \psi \|_{L^2}$, this verifies the claim.
\end{proof}

\section{Weighted refined decoupling estimates}\label{sec:implications}
In this section, we verify the weighted refined decoupling estimate,  \Cref{thm:LpWeightedDecouplingConcentratedRefined}, and explore its implications. We first recall the refined decoupling estimate.
\begin{theorem}[Refined decoupling] \label{thm:refined_decouplingDuGeneral}
	Let $2\leq m \leq d$. Let $f = \sum_{T \in \W} f_T$ be a sum of wave packets with $(ER^{-1/2}, m)$-concentrated frequencies for some $E\geq 1$. Suppose that $\{ Q\}$ is a collection of $R^{1/2}$-cubes in $B_R$, such that each $Q$ is contained in at most $M\geq 1$ of the dilated tubes $\{ 2T\}_{T\in \W}$. Then
	\[ \| f \|_{L^{p_m}(\cup Q)} \less E^{\frac{d-m}{m+1}} M^{\frac{1}{m+1}} \Big( \sum_{T\in \W} \| f_T \|_{L^{p_m}}^{p_m} \Big)^{1/p_m}. \]
\end{theorem}
For the case $m=d$, \Cref{thm:refined_decouplingDuGeneral} was obtained in \cite{GIOW}. When $E=1$, the result coincides with \cite[Theorem 1.2(a)]{du2023weighted}. For general $E\geq 1$, it follows from \cite[Theorem 1.2(c)]{du2023weighted} by taking $r=ER^{-1/2}$, $\alpha = d$ and $p=p_m$. 

\begin{proof}[Proof of  \Cref{thm:LpWeightedDecouplingConcentratedRefined}]
	Let $2\leq p \leq p_m$. 
	We may assume that $\Big( \sum_{T\in \W} \| f_T \|_{L^p}^p \Big)^{\frac{1}{p}} = 1$. For dyadic $\lambda \in [R^{-100d} , 1]$, let $\W_\lambda := \{ T\in \W: \| f_T \|_p \sim \lambda \}$ and $f_\lambda = \sum_{T\in \W_\lambda} f_T$. Then up to a negligible term, there exists a dyadic $\lambda$ such that 
	\[ 
	\| f\|_{L^p(Y)} \les \log R   \| f_\lambda \|_{L^p(Y)}.
	\]
	
	Since $f_T$ is a wave packet, we have \[ \| f_T \|_{L^2} \sim |T|^{\frac{1}{2}-\frac{1}{p}} \| f_T\|_{L^p}\] 
	and therefore,
	\begin{equation}\label{eqn:wavepacketNorm}
		\bigg( \sum_{T\in \W_\lambda} \| f_T\|_{L^2}^2  \bigg)^\frac{1}{2} \sim  (|\W_\lambda| |T|)^{\frac{1}{2}-\frac{1}{p}} \bigg( \sum_{T\in \W_\lambda}  \| f_T\|_{L^p}^p \bigg)^{\frac{1}{p}} 
	\end{equation}
	for all $2\leq p\leq \infty$. 
	By H\"{o}lder, 
	\[ \| f_\lambda \|_{L^p(Y)} \leq \| f_\lambda \|_{L^2(Y)}^{(m+1)(\frac{1}{p}-\frac{1}{p_m})} \| f_\lambda\|_{L^{p_m}(\cup Q)}^{(m+1)(\frac{1}{2}-\frac{1}{p})}. \]
	Applying \Cref{thm:LpWeightedDecouplingConcentrated} to $\| f_\lambda \|_{L^2(Y)}$ and \Cref{thm:refined_decouplingDuGeneral} to $\| f_\lambda \|_{L^{p_m}(\cup Q)}$, and subsequently invoking \eqref{eqn:wavepacketNorm}, we arrive at the claimed estimate.
\end{proof}

Following  \cite{GIOW, du2023weighted,  KimWeightedDec}, we derive a weighted estimate from \Cref{thm:refined_decouplingDuGeneral}.
\begin{theorem}\label{thm:b}
Let	$1\leq m \leq d$. 	Let $\W \subset \T(R)$ and $f= \sum_{T\in \W} f_T$ be a sum of wave packets with $(ER^{-1/2},m)$-concentrated frequencies for some $E\geq 1$. Let $Y\subset B_R$ be a union of unit cubes in $\R^d$. Then
	\[  \| f\|_{L^2(Y)}  \less E^{\frac{d-m}{m+1}} \Big( \max_{T\in \W} \frac{|Y\cap 2T|}{|T|} \Big)^{1/(m+1)} \Big( \sum_{T\in \W} \| f_T \|_{L^{2}}^{2} \Big)^{\frac{1}{2}}.\]
\end{theorem}
\begin{proof}
When $m=1$, the proof is elementary. In this case, we may write $\W$ as a union of $\W_\theta= \W \cap \T_\theta$ for $O(E^{d-1})$ caps $\theta$. Thus, by Cauch-Schwarz, 
\[ \| f\|_{L^2(Y)}^2 \les E^{d-1} \sum_{\theta} \| \sum_{T\in \W_\theta} f_T \|_{L^2(Y)}^2 \less E^{d-1}    \sum_{\theta}  \sum_{T\in \W_\theta} \|f_T \|_{L^2(Y\cap 2T)}^2, \]
where we used the fact that $f_T$ is negligible outside of $2T$ and $\{2T\}$ has $\less 1$ overlap. Since $f_T$ is a wave packet, we have 
\[ \| f_T \|_{L^2(Y\cap 2T)}^2 \leq  | Y\cap 2T| \| f_T \|_{L^\infty}^2 \sim \frac{| Y\cap 2T| }{|T|} \| f_T\|_{L^2}^2. \]
Combining these bounds yields the claim for $m=1$. 

Now let $2\leq m \leq d$. Arguing as in the proof of  \Cref{thm:LpWeightedDecouplingConcentratedRefined}, we may assume that $\| f_T \|_2$ is comparable to a fixed constant for every $T\in \W$. By dyadic pigeonholing, there exist $M$ and a collection $\{ Q\}$ of $R^{1/2}$-cubes in $B_R$, such that each $Q$ is contained in $\sim M$ of the dilated tubes $\{ 2T\}_{T\in \W}$ and  
	$$\| f\|_{L^2(Y)} \less  \| f\|_{L^2(Y_1)}, \;\; Y_1 = \cup_{Q\in \mathcal{Q}} Y\cap Q.$$ 
	By H\"{o}lder and  \Cref{thm:refined_decouplingDuGeneral}, 
	we have
	\begin{equation}\label{eqn:L2Lp}
		\| f\|_{L^2(Y_1)} \leq |Y_1|^{\frac{1}{m+1}} \| f\|_{L^{p_m}(Y_1)}\less E^{\frac{d-m}{m+1}} (|Y_1|M)^{\frac{1}{m+1}} \Big( \sum_{T\in \W} \| f_T \|_{L^{p_m}}^{p_m} \Big)^{1/p_m}. 
	\end{equation}
	
	By \eqref{eqn:wavepacketNorm} and \eqref{eqn:L2Lp}, we obtain
	\begin{equation}\label{eqn:L2weighted}
		\| f\|_{L^2(Y_1)} \less E^{\frac{d-m}{m+1}} \Big(\frac{|Y_1|M}{|T|\# \W}\Big)^{\frac{1}{m+1}} \Big( \sum_{T\in \W} \| f_T \|_{L^{2}}^{2} \Big)^{\frac{1}{2}}.
	\end{equation}
	
	Finally, it suffices to observe that 
	\begin{align*}
		|Y_1| M  \sim  \sum_{Q\in \Q} \sum_{T\in \W: Q \subset 2T} |Y\cap Q|   \leq \sum_{T\in \W} |Y\cap 2T| \leq \# \W \max_{T\in \W} |Y\cap 2T|.
	\end{align*}
\end{proof}

The following is an immediate consequence of \Cref{thm:LpWeightedDecouplingConcentratedRefined}.
\begin{theorem}\label{thm:LpWeightedDecoupling}
	Let $2\leq m \leq d$  and $0\leq \alpha \leq d$. Let $f = \sum_{T \in \W} f_T$ be a sum of wave packets with $(R^{-1/2}, m)$-concentrated frequencies. Then for $2\leq p \leq p_m$, 
	\begin{align*}
		\| f \|_{L^p(Y)} \less \gamma_{\alpha,R}(Y)^{\frac{m+1}{m}(\frac{1}{p}-\frac{1}{p_m})} R^{-\frac{1}{2} ( d+1 -  \frac{m+1}{m} \alpha)(\frac{1}{p}-\frac{1}{p_m})} 
		M^{\frac{1}{2}-\frac{1}{p}}   \bigg( \sum_{T\in \W} \| f_T \|_{L^p}^p \bigg)^{\frac{1}{p}}.
	\end{align*}
\end{theorem}

Via a two-step decoupling argument, \Cref{thm:LpWeightedDecoupling} implies the following more general result. 
\begin{theorem}\label{thm:linearWeightedDecoupling}
	Let $2\leq m \leq d$ and $0\leq \alpha \leq d$. Let $f = \sum_{T \in \W} f_T$ be a sum of wave packets with $(r, m)$-concentrated frequencies. Then for any $p_d\leq p\leq p_m$, 
	\begin{align*}
		\| f\|_{L^p(Y)} \less \gamma_{\alpha,r^{-2}}(Y)^{\frac{m+1}{m}(\frac{1}{p}-\frac{1}{p_m})} &r^{(d+1 - \frac{m+1}{m}\alpha)(\frac{1}{p}-\frac{1}{p_m})} \\ &(r^2 R)^{\frac{d+1}{2}(\frac{1}{p_d} -\frac{1}{p})}  
	M^{\frac{1}{2}-\frac{1}{p}}   \bigg( \sum_{T\in \W} \| f_T \|_{L^p}^p \bigg)^{\frac{1}{p}}. 
	\end{align*}
\end{theorem}

The proof that \Cref{thm:LpWeightedDecoupling} $\implies$ \Cref{thm:linearWeightedDecoupling} is identical to that of the implication  Theorem 1.2(b) $\implies$  Theorem 1.2(c) in \cite{du2023weighted}. Therefore, we omit the proof and refer the reader to \cite{du2023weighted}.

Next, we establish a weighted  $L^2$ estimate from \Cref{thm:linearWeightedDecoupling}, which will be applied to the Falconer distance set problem. For a non-negative measurable function $H$ on $\R^d$, we denote $H(E) = \int_E H$ and 
\[ \gamma_{\alpha,R}(H) :=   \sup_{1\leq \rho \leq R} \sup_{B_\rho} \rho^{-\alpha} H(B_\rho).\]
We say that $H$ is essentially constant at scale $s$ if $H(x) \sim H(y)$ whenever $|x-y|\leq s$. 

\begin{theorem}\label{thm:measureVer}
	Let $2\leq m \leq d$. Let $f = \sum_{T \in \W} f_T$ be a sum of wave packets with $(r, m)$-concentrated frequencies. 
	Let $H$ be a non-negative measurable function on $\R^d$ which is essentially
	constant at scale $1$ and satisfies $0\leq H\les 1$. Assume that $\gamma_{\alpha,r^{-2}}(H) \les 1$. Then for any $p_d\leq p\leq p_m$, 
	\[ \| f\|_{L^2(H)} \less \big(  |T|^{-1} \max_{T\in \W} H(2T) \big)^{\frac{1}{2}-\frac{1}{p}} r^{(d+1 - \frac{m+1}{m}\alpha)(\frac{1}{p}-\frac{1}{p_m})} (r^2 R)^{\frac{d+1}{2}(\frac{1}{p_d} -\frac{1}{p})} \| f\|_{L^2}.  \]
\end{theorem}
\begin{proof}
We may assume that $\| f_T \|_{L^2}$ is comparable to a fixed constant for every $T\in \W$, arguing as in the proof of \Cref{thm:LpWeightedDecouplingConcentratedRefined}.

Since $H$ is essentially constant at scale 1, it is comparable to a function which is constant on each lattice unit cube. Thus, without loss of generality, we may assume that $H$ is constant on lattice unit cubes. For $\la>0$, let $Y_\lambda = \{ x: H(x) \sim \lambda\}$ denote the union of lattice unit cubes where $H$ is comparable to $\la$. By dyadic pigeonholing, there exists  a dyadic $\lambda \in [R^{-200d}, 1]$ such that 
    \[ \| f\|_{L^2(H)} \less \lambda^{1/2} \| f\|_{L^2(Y_\la)} + R^{-100d} \| f\|_{L^2}. \]
    The term involving $R^{-100d}$ is negligible and will be omitted in what follows.

It suffices to prove the estimate under the additional assumption that
$Y_\lambda$ is contained in a fixed $R$-cube $B_R$ and each dilated tube $2T$ intersects $B_R$ for each $T\in \W$. Indeed, after partitioning
$\mathbb R^d$ into lattice $R$-cubes, one may apply the local estimate on each
cube, keeping only the wave packets contributing to $B_R$. Summing up these local estimates yields the claimed bound. 

We partition $Y_\la$ by $R^{1/2}$-cubes. Then by dyadic pigeonholing, there exists $M\geq 1$ and a collection of $R^{1/2}$-cubes $\{ Q\}$ such that each $Q$ is contained  in $\sim M$ dilated tubes $\{ 2T: T\in \W\}$ and 
  \[ \| f\|_{L^2(Y_\la)} \less \| f\|_{L^2(Y_{\la,M})}, \; \text{ where } \; Y_{\la,M} = Y_\la \cap (\cup Q). \]

 Note that 
\[ |Y_{\la,M}| \les \la^{-1} H(Y_{\la,M}).\]
  By H\"older,
  \begin{align*}
  	\| f\|_{L^2(H)} \less \lambda^{\frac{1}{2}} |Y_{\la,M}|^{\frac{1}{2}-\frac{1}{p}} \| f\|_{L^p(Y_{\la,M})}  \les \la^{\frac{1}{p}} H(Y_{\la,M})^{\frac{1}{2}-\frac{1}{p}} \| f\|_{L^p(Y_{\la,M})} .
  \end{align*}

By \Cref{thm:linearWeightedDecoupling}, we have 
\begin{align*}
	\| f\|_{L^p(Y_{\la,M})}  \less  &\gamma_{\alpha,r^{-2}}(Y_{\la,M})^{\frac{m+1}{m}(\frac{1}{p}-\frac{1}{p_m})}  \\ &r^{(d+1 - \frac{m+1}{m}\alpha)(\frac{1}{p}-\frac{1}{p_m})} (r^2 R)^{\frac{d+1}{2}(\frac{1}{p_d} -\frac{1}{p})}   M^{\frac{1}{2}-\frac{1}{p}}   \bigg( \sum_{T\in \W} \| f_T \|_{L^p}^p \bigg)^{\frac{1}{p}}.
\end{align*}

Note that 
\[  \gamma_{\alpha,r^{-2}}(Y_{\la,M}) \les \la^{-1} \gamma_{\alpha,r^{-2}}(H) \les \la^{-1}. \]
Combining these estimates and using \eqref{eqn:wavepacketNorm}, we get
\begin{equation}\label{eqn:a}
	\begin{split}
		\| f\|_{L^2(H)} \less  \la^{\frac{1}{p}-\frac{m+1}{m}(\frac{1}{p}-\frac{1}{p_m})} \Big(\frac{H(Y_{\la,M}) M}{|\W| |T| }\Big)^{\frac{1}{2}-\frac{1}{p}} 
		r^{(d+1 - \frac{m+1}{m}\alpha)(\frac{1}{p}-\frac{1}{p_m})} (r^2 R)^{\frac{d+1}{2}(\frac{1}{p_d} -\frac{1}{p})}  \| f\|_{L^2}.
	\end{split}
\end{equation}

The exponent of $\lambda$ is
\[
\frac1p-\frac{m+1}{m}\left(\frac1p-\frac1{p_m}\right)
=
\frac1m\left(\frac{m-1}{2}-\frac1p\right),
\]
which is non-negative for the relevant range of $p$. Hence $\la^{\frac{1}{p}-\frac{m+1}{m}(\frac{1}{p}-\frac{1}{p_m})} \leq 1$. Note that
\[ H(Y_{\la,M}) M \sim \sum_{Q} \sum_{T\in \W: Q\subset 2T} H(Y_{\la,M} \cap Q) \leq  \sum_{T\in \W}  H(Y_{\la,M} \cap 2T)\leq | \W| \max_{T\in \W} H(2T). \]
Combining these observations with \eqref{eqn:a} finishes the proof.
\end{proof}

\section{Applications to the Falconer distance set problem}\label{sec:Falconer}
In this section, we verify \Cref{thm:Falconer}. We first describe the framework proposed in \cite{du2023weighted}. The analysis of the good measure part essentially reduces to the following problem:

Let $d\geq 3$ and $\frac{d}{2}< \alpha< \frac{d+1}{2}$. 
Let $\mu$ be a probability measure supported on the unit ball $B_1(0)$ such that 
\begin{equation}\label{eqn:measureAlpha}
	\sup_{0< \rho \leq 1 } \sup_{B_\rho} \rho^{-\alpha} \mu(B_\rho) \les 1.
\end{equation}
Let $m=\lfloor d/2 \rfloor+1$ so that $m-1<\alpha \leq m$. Suppose that 
\begin{itemize}
	\item $\W \subset \T(R)$ is a collection of tubes intersecting $B_R(0)$ associated with the unit sphere  $S=\S^{d-1}$.
	\item $g(x)= f(Rx)$, where $f = \sum_{T\in \W} f_T$ has $(r,m)$-concentrated frequencies for some $r\in [R^{-1/2},1]$. Thus, $g$ is a sum of wave packets associated with tubes $T_{R^{-1}} := R^{-1}T$ of dimensions $R^{-1/2} \times \cdots \times R^{-1/2} \times 1$.
	\item Let $\mu_{R^{-1}} = \mu* \eta_{R^{-1}}$, where $\eta_{R^{-1}} (x)= R^{d}(1+R|x|)^{-100d}$. Each $T\in \W$ is good in the sense that   
	\[   \mu_{R^{-1}} (2 T_{R^{-1}}) \less R^{-\alpha/2} r^{-(\alpha-(m-1))}.\footnote{We have slightly modified the definition from \cite{du2023weighted} by replacing $\mu$ with the regularized version $\mu_{R^{-1}}$, which makes the application of refined decoupling cleaner. It appears that this change does not affect the treatment of the bad tubes.} \]
\end{itemize}
 Under this setup, we would like to determine the largest exponent $\Gamma(d,\alpha)$ for which  
\begin{equation}\label{eqn:boundFalconer}
	\| g \|_{L^2(\mu)}^2 \less R^{-\Gamma(d,\alpha) } R^{d-1} \| g \|_{L^2}^2,
\end{equation}
where the implicit constant does not depend on $r\in [R^{-1/2},1]$.

Given the bound \eqref{eqn:boundFalconer}, a threshold for the Falconer distance set problem can be obtained by solving $\alpha + \Gamma(d,\alpha) > d$ for $\alpha$; see \cite[Section 4.3]{du2023weighted}.  In that paper, it was shown that \eqref{eqn:boundFalconer} holds with 
\begin{equation}\label{eqn:duGamma}
	\Gamma(d,\alpha) = \begin{cases}
		\frac{d\alpha}{d+1} & d\geq 4\\ 
		\alpha - \frac{\alpha^2}{6} & d=3.
	\end{cases}
\end{equation}

We present a weighted $L^2$ estimate for the above problem.
\begin{theorem}\label{thm:decFalconer}
	Let $2\leq m \leq d$. Let $g$ and $\mu$ be as above. Then for any $p_d\leq p\leq p_m$, 
	\begin{align*}
		\| g\|_{L^2(\mu)} \less R^{\frac{d-\alpha}{p}}  \Big(R^{\frac{d-1}{2}} \max_{T\in \W} \mu_{R^{-1}}(2T_{R^{-1}}) \Big)^{\frac{1}{2}-\frac{1}{p}}  r^{(d+1 - \frac{m+1}{m}\alpha)(\frac{1}{p}-\frac{1}{p_m})} (r^2 R)^{\frac{d+1}{2}(\frac{1}{p_d} -\frac{1}{p})} \| g\|_{L^2}.
	\end{align*}
\end{theorem}
A similar estimate with the factor $r^{(d+1 - \frac{m+1}{m}\alpha)(\frac{1}{p}-\frac{1}{p_m})}$ replaced by $r^{(d - \alpha)(\frac{1}{p}-\frac{1}{p_m})}$ is implicit in \cite{du2023weighted}. 

Before giving the proof, we explore the implications  of \Cref{thm:decFalconer}. For $d\geq 4$, \Cref{thm:decFalconer} at $p=p_d$ recovers \eqref{eqn:duGamma}, whereas for $d=3$, it yields a slight improvement.   \Cref{thm:Falconer} then follows by solving $\alpha+\Gamma(3,\alpha) > 3$ for $\alpha$ and applying the following result.
\begin{prop}\label{cor:decFalconer}
	For $d=3$,  \eqref{eqn:boundFalconer} holds with
	\[ \Gamma(3,\alpha) = \begin{cases}
		\frac{3\alpha}{4} & \frac{3}{2} < \alpha \leq \frac{14}{9}, \\
		\frac{8\alpha}{3(2+\alpha)} & \alpha > \frac{14}{9}.
	\end{cases} \]
\end{prop}
\begin{proof}
	Let $d=3$ and $m=2$. Applying \Cref{thm:decFalconer} for  $p=p_d=4$ and $p=p_m = 6$ yields 
	\begin{equation}\label{eqn:p4}
		\| g\|_{L^2(\mu)} \less R^{\frac{3-\alpha}{4}}  (R^{1-\frac{\alpha}{2}} r^{1-\alpha} )^{\frac{1}{2}-\frac{1}{4}} r^{(4 - \frac{3}{2}\alpha)(\frac{1}{4}-\frac{1}{6})} \| g\|_{L^2} = R^{-\frac{3}{8}\alpha} r^{-\frac{3}{8}(\alpha-\frac{14}{9})} R \| g\|_{L^2}.
	\end{equation}
	and 
	\begin{equation}\label{eqn:p6}
		\| g\|_{L^2(\mu)} \less R^{\frac{3-\alpha}{6}}  (R^{1-\frac{\alpha}{2}} r^{1-\alpha} )^{\frac{1}{2}-\frac{1}{6}} (r^2 R)^{2(\frac{1}{4} - \frac{1}{6})} \| g\|_{L^2} = R^{-\frac{1}{3}\alpha} r^{\frac{2-\alpha}{3}} R \| g\|_{L^2}.
	\end{equation}
	When $\frac{3}{2}<\alpha\leq \frac{14}{9}$, \eqref{eqn:p4} yields \eqref{eqn:boundFalconer}  with $\Gamma(3,\alpha) = \frac{3\alpha}{4}$. 
	
	When $\alpha>\frac{14}{9}$, \eqref{eqn:p4} and \eqref{eqn:p6} are effective for large and small values of $r$, respectively. Balancing these bounds, we obtain
	\[ 	\| g\|_{L^2(\mu)} \less R^{-\frac{4\alpha}{3(2+\alpha)}} R \| g\|_{L^2}, \]
	for all $r\in [R^{-1/2},1]$; here the scale $r= R^{-\frac{\alpha}{2+\alpha}}$ is critical.  This yields  \eqref{eqn:boundFalconer} with $\Gamma(3,\alpha) = \frac{8\alpha}{3(2+\alpha)}$. 
\end{proof}

It remains to verify \Cref{thm:decFalconer}.
\begin{proof}[Proof of \Cref{thm:decFalconer}]
	Since $\hat{g}$ is supported on $B_{2R}(0)$, we have $|g|^2 \les |g|^2*\eta_{R^{-1}}$, which yields $\| g\|_{L^2(\mu)} \les \| g\|_{L^2(\mu_{R^{-1}} )}$. Let 
	\[ H(x) = R^{-(d-\alpha)}\mu_{R^{-1}}  (R^{-1}x).\]
	Then by \eqref{eqn:measureAlpha}, $0\leq H\les 1$ and $H$ satisfies $\gamma_{\alpha,R}(H) \les 1$.
	Moreover, $H$ is essentially constant at scale 1 since $\mu_{R^{-1}}$ is essentially constant at scale $R^{-1}$. Therefore, we may apply  \Cref{thm:measureVer}.
	
	Note that 
	\[ \frac{ \| g\|_{L^2(\mu_{R^{-1}} )}^2}{\| g\|_{L^2}^2} = \frac{ R^{-\alpha} \| f\|_{L^2(H)}^2}{ R^{-d} \| f\|_{L^2}^2}. \]
	Applying \Cref{thm:measureVer}, we obtain
	\[  \frac{ \| g\|_{L^2(\mu_{R^{-1}} )} }{ \| g\|_{L^2}} \less R^{\frac{d-\alpha}{2}}  \big( \max_{T\in \W} |T|^{-1} H(2T) \big)^{\frac{1}{2}-\frac{1}{p}} r^{(d+1 - \frac{m+1}{m}\alpha)(\frac{1}{p}-\frac{1}{p_m})} (r^2 R)^{\frac{d+1}{2}(\frac{1}{p_d} -\frac{1}{p})}. \]
	Since
	\[ |T|^{-1}  H(2T)  \sim \frac{1}{R^{\frac{d+1}{2}}} R^{-(d-\alpha)}\int_{2T}  \mu_{R^{-1}}  (R^{-1}x) dx = R^{-(d-\alpha)} R^\frac{d-1}{2}  \mu_{R^{-1}}(2T_{R^{-1}}),\]
	combining these estimates finishes the proof.
	
\end{proof}

\section{A multilinear weighted $L^2$ estimate}\label{sec:refined}
In this section, we derive a multilinear weighted $L^2$ estimate from refined decoupling and multilinear Kakeya estimates. 

We say that the caps $U_1,\ldots,U_m\subset \S^{d-1}$ are $\nu$-transverse, for some $0<\nu\leq 1$, if
\[ 
	|v_1\wedge v_2\wedge\cdots\wedge v_m|\geq \nu
\]
for every choice of $v_j\in U_j$. 

We recall the Bennett-Carbery-Tao multilinear Kakeya estimate \cite[Theorem 5.1]{BCT}, in the form stated in \cite[Theorem 4.4]{DGLZ}. 
\begin{theorem}[Multilinear Kakeya]\label{thm:multiKakeya} 	
	Let $2\leq m \leq d$ and let $U_1,\ldots,U_m\subset \S^{d-1}$ be $\nu$-transverse caps. 
	For each $j=1,\cdots,m$, suppose that $\{ l_{j,a} \}_{a=1}^{N_j}$ is a collection of lines in $\R^{d}$ whose directions lie in $U_j$. Let $T_{j,a}$ denote the 1-neighborhood of $l_{j,a}$. Then for any $\epsilon>0$ and any cube $Q_L$ of side length $L\geq 1$,
	\[ \int_{Q_L} \prod_{j=1}^m \bigg( \sum_{a=1}^{N_j} \chi_{T_{j,a}}\bigg)^{\frac{1}{m-1}} \leq   C_\epsilon L^{\epsilon}  \nu^{-O(1)}  \prod_{j=1}^m {N_j}^{\frac{1}{m-1}}.\]
\end{theorem}

\Cref{thm:multiKakeya} has the following consequence.
\begin{prop} \label{cor:multiKakeya}
	Let $2\leq m \leq d$ and let $U_1,\ldots,U_m\subset \S^{d-1}$ be $\nu$-transverse caps. 	For each $j=1,\cdots,m$, let $\W_j \subset \T(R)$ be a collection of tubes whose directions lie in $U_j$.  Let $\mathcal{Q} = \{ Q \}$ be a collection of  disjoint $R^{1/2}$-cubes contained in $B_R$ such that each $Q$ is contained in $\sim M_j$ dilated tubes $\{ 2T\}_{T\in \W_j}$ for each $1\leq j\leq m$. Then 
	\[ \# \mathcal{Q} \less \nu^{-O(1)} \prod_{j=1}^m \left( \frac{\# \W_j}{M_j} \right)^{1/(m-1)}. \]
\end{prop}

We derive a multilinear estimate by combining the refined decoupling estimate with Proposition \ref{cor:multiKakeya}, adapting the approach from \cite{DGLZ} and \cite{DemRefined}.
\begin{theorem}\label{thm:mulc_simplified}
		Let $2\leq m \leq d$ and let $U_1,\ldots,U_m\subset \S^{d-1}$ be $\nu$-transverse caps.	For each $j=1,\cdots,m$, let $\W_j \subset \T(R)$ be a collection of tubes whose directions lie in $U_j$. Let $f_j = \sum_{T\in \W_j} f_T$ be a sum of wave packets with $(ER^{-1/2},m)$ concentrated frequencies for some $E\geq 1$.  Let $\mathcal{Q}$ be a collection of disjoint $R^{1/2}$-cubes in $B_R$ and $Y \subset \cup_{Q\in \mathcal{Q}} Q$ be a union of unit cubes in $\R^d$. Then 
	\begin{align*}
		\Big(\sum_{Q \in \mathcal{Q}}
		&\prod_{j=1}^m \| f_j\|_{L^2(Y\cap Q)}^{2/m} \Big)^{1/2} \\ &\less \nu^{-O(1)}  E^{\frac{d-m}{m+1}} 
		 \Big( |Y| \max_{Q\in \mathcal{Q}} |Y\cap Q| ^{m-1}  
	R^{-\frac{m(d+1)}{2}} \Big)^{\frac{1}{m(m+1)}} \prod_{j=1}^m   \| f_j \|_{L^2}^{\frac{1}{m}}.
	\end{align*}
\end{theorem}
\begin{proof}
	Arguing as in the proof of  \Cref{thm:LpWeightedDecouplingConcentratedRefined},  we may assume that, for each $1\leq j\leq m$, $\| f_T \|_2$ is comparable to a fixed constant (which may depend on $j$) for every $T\in \W_j$. 
	
	By dyadic pigeonholing, we may find dyadic  $M_1,M_2,\cdots,M_m\geq 1$ and a subcollection $\mathcal{Q}_1 \subset \mathcal{Q}$ such that
	\[	\Big(\sum_{Q \in \mathcal{Q}}
	\prod_{j=1}^m \| f_j\|_{L^2(Y\cap Q)}^{2/m} \Big)^{1/2} \less \Big(\sum_{Q \in \mathcal{Q}_1}
	\prod_{j=1}^m\| f_j\|_{L^2(Y\cap Q)}^{2/m} \Big)^{1/2}  \]
	and  each $Q\in \mathcal{Q}_1$ is contained in $2T$ for $\sim M_j$ tubes $T\in \W_j$ for every $j=1,2,\cdots, m$. 
	
	Let $Y_1 = \cup_{Q\in \mathcal{Q}_1} Y\cap Q$. By H\"{o}lder and the proof of \eqref{eqn:L2weighted}, we have 
	\begin{align*}
		\Big(\sum_{Q \in \mathcal{Q}_1}
		\prod_{j=1}^m \| f_j\|_{L^2(Y\cap Q)}^{2/m} \Big)^{1/2}   &\leq	\prod_{j=1}^m \| f_j\|_{L^2(Y_1)}^{1/m}\\
		&\less E^{\frac{d-m}{m+1}} \prod_{j=1}^m  \Big(\frac{|Y_1|M_j}{|T|\# \W_j}\Big)^{\frac{1}{m(m+1)}}  \Big( \sum_{T\in \W_j} \| f_T \|_{L^2}^2 \Big)^{\frac{1}{2m}}.
	\end{align*}
	
	Note that $ \sum_{T\in \W_j} \| f_T \|_{L^2}^2  \sim \| f_j \|_{L^2}^2$.
	By Proposition \ref{cor:multiKakeya}, we have
	\[ 
	\prod_{j=1}^m  \frac{M_j}{\# \W_j}  \less \nu^{-O(1)}   (\# \mathcal{Q}_1)^{-(m-1)}. \]

	Combining these bounds, we get 
	\begin{align*}
		\Big(\sum_{Q \in \mathcal{Q}}
		&\prod_{j=1}^m \| f_j\|_{L^2(Y\cap Q)}^{2/m} \Big)^{1/2}
		\\ &\less \nu^{-O(1)}  E^{\frac{d-m}{m+1}} \Big( |Y_1|^m 
		(\# \mathcal{Q}_1)^{-(m-1)}  R^{-\frac{m(d+1)}{2}} \Big)^{\frac{1}{m(m+1)}} \prod_{j=1}^m  \| f_j \|_{L^2}^{\frac{1}{m}}.
	\end{align*}
	
	The claimed bound follows since 
	\begin{align*}
	|Y_1|^m 
	(\# \mathcal{Q}_1)^{-(m-1)}  
		=	 |Y_1|
		\Big(\frac{\sum_{Q\in \mathcal{Q}_1}  |Y\cap Q| }{\# \mathcal{Q}_1}\Big)^{m-1} \leq |Y|    \max_{Q\in \mathcal{Q}_1}  |Y\cap Q|^{m-1}. 
	\end{align*}
\end{proof}

\section{Proof of \Cref{thm:LpWeightedDecouplingConcentrated}}\label{sec:proofMain}
To facilitate the induction argument, we prove a slightly more general version of \Cref{thm:LpWeightedDecouplingConcentrated} with an additional parameter $E\geq 1$. 
\begin{theorem}\label{thm:LpWeightedDecouplingConcentratedGeneral}
	Let $2\leq m \leq d$. Let $f = \sum_{T \in \W} f_T$ be a sum of wave packets with $(ER^{-1/2}, m)$-concentrated frequencies for some $E\geq 1$, and let $Y\subset B_R$ be a union of unit cubes.
	Suppose  $0\leq \alpha, \beta \leq d$ satisfy $\beta \geq \alpha +m - d$. Then for each $0<\eps<1$, there exists an absolute constant $C_\eps$ which may depend on $\eps$ and $d$  for which 
	\begin{align*}
	\| f \|_{L^2(Y)} \leq C_\eps R^\eps E^{\frac{d-m}{m+1}} \mathcal{A}_{d,m,\alpha,\beta}(R, Y) ^{\frac{1}{m(m+1)}} 
	\| f \|_{L^2}.
	\end{align*}
\end{theorem}
\Cref{thm:LpWeightedDecouplingConcentrated} is then a special case of \Cref{thm:LpWeightedDecouplingConcentratedGeneral} for $E=1$. 

The estimate is trivial if $Y$ is empty. Thus, we assume that $Y$ contains at least one unit cube. Accordingly, we have 
\begin{equation}\label{eqn:lowerBoundQuantities}
	 \gamma_{\alpha,R^{1/2}}(Y)  \geq 1, \;\;  1\leq |Y| \leq \overline{\gamma_{\beta,R}}(Y) R^{\beta}, \;\; \sup_{B_{R^{1/2}}} |Y\cap B_{R^{1/2}}| \geq 1.
	\end{equation}
Moreover, we may assume that $2T\cap Y$ is non-empty for each $T\in \W$ since the contribution of $f_T$ outside $2T$ is negligible. 

For a given $0<\eps<1$, we choose $0<\eps_1 \ll \eps$ sufficiently small relative to $\eps$.
For $K=R^{\eps_1}$, consider a canonical covering of $N_{K^{-2}} (S)$ by parallelepipeds $\tau$ of dimensions $K^{-1}\times \cdots \times K^{-1}\times K^{-2}$.
For each $T\in \W$, there exist $\sim 1$ rectangles $\tau$ containing $2\theta(T)$, where $\theta(T)$ denotes the $\theta$ for which $T\in \T_\theta$. We fix one such $\tau=\tau(T)$ for each $T\in \W$ and let $\W_\tau= \{ T\in \W: \tau(T) = \tau\}$, so that 
\[ f=\sum_\tau f_\tau, \;\; \text{where} \;\; f_\tau = \sum_{T \in \W_\tau} f_T. \]

Let $N$ denote the number of caps $\tau$; thus $N \sim K^{d-1}$.  For each lattice $K^2$-cube $B$ intersecting $Y$, define 
\[ S(B) = \{ \tau :  \| f_\tau \|_{L^2(B\cap Y)} \geq \frac{1}{ 2 N } \| f\|_{L^2(B\cap Y)} \}.  \]
The definition of $S(B)$ ensures that for each $K^2$-cube $B$,
\[ \| f\|_{L^2(B\cap Y)} \sim \| \sum_{\tau\in S(B)} f_\tau \|_{L^2(B\cap Y)}. \]

We say that a $K^2$-cube $B$  intersecting $Y$ is narrow if there exists a $(m-1)$-dimensional subspace $V\subset \R^d$ such that for every $\tau\in S(B)$, 
\[ \angle(G(\tau),V) \leq K^{-1},\]
where $G: S \to \S^{d-1}$ is a choice of unit normal map and $\angle(G(\tau),V)$ denotes the smallest angle between any non-zero vectors $u\in G(\tau)$ and $v\in V$. Otherwise, we say that $B$ is broad. If $B$ is broad, then there exist $\tau_1,\tau_2,\cdots,\tau_m \in S(B)$ such that
\begin{equation}\label{eqn:transverse}
	|v_1\wedge v_2\wedge\cdots\wedge v_m| \ges K^{-(m-1)},
\end{equation}
for any $v_i \in G(\tau_i)$.

Let $\Bb$ denote the collection of broad $K^2$-cubes and $\Bn$ denote the collection of narrow $K^2$-cubes. Then we have 
\[ Y= \Yb \bigsqcup \Yn, \]
where $\Yb$ is the union of $B\cap Y$ over $B\in \Bb$ and $\Yn$ is defined similarly. We say that we are in the broad case if $\| f\|_{L^2(\Yb)} \geq \| f\|_{L^2(\Yn)}$. Otherwise, we  say that we are in the narrow case. We handle each case in the following subsections.

\subsection{Broad case}
In this subsection, 
we write $A \less B$ for estimates of the form $A \leq K^{O(1)}B$; such a loss is harmless as long as $K^{O(1)} \leq R^{\eps/2}$.

In the broad case,  we have $\| f\|_{L^2(Y)} \les \| f\|_{L^2(\Yb)}$.  
For each $B\in \Bb$, there exist $\tau_1,\tau_2,\cdots,\tau_m \in S(B)$ satisfying \eqref{eqn:transverse} for any $v_i \in G(\tau_i)$. We fix such an $m$-tuple and denote $\bar{\tau}(B) = (\tau_1,\cdots,\tau_m)$. Let 
\[ \mathfrak T = \{ \bar{\tau}(B) : B \in \Bb \}. \]
Since $\# \mathfrak T \les K^{O(1)}$, by dyadic pigeonholing, there exist $\bar{\tau}=(\tau_1,\cdots,\tau_m) \in \mathfrak T$ and a sub-collection  $\Bb^1 \subset \Bb$ such that $\bar{\tau}(B) = \bar{\tau}$ for all $B\in \Bb^1$ and 
\[   \| f\|_{L^2(\Yb)}\less  \| f\|_{L^2(\Yb^1)}, \text{ where } \; \Yb^1 = \bigcup_{B\in \Bb^1} B\cap Y. \]

Let $B\in \Bb^1$. Since $\tau_j \in S(B)$ for each $j$ and $N\sim K^{d-1} \less 1$, we have 
\[ 
	\| f \|_{L^2(Y\cap B)}^2 \less  \prod_{j=1}^m \| f_{\tau_j} \|_{L^2(Y\cap B)}^{2/m}.
\]
Let $Q\subset B_R$ be a lattice $R^{1/2}$-cube. Summing the estimate over $B\in\Bb^1$ contained in $Q$ using H\"{o}lder, we obtain 
\begin{align*}
	\| f \|_{L^2(\Yb^1 \cap Q)}^2 \less \prod_{j=1}^m \| f_{\tau_j} \|_{L^2(\Yb^1\cap Q)}^{2/m}. 
\end{align*} 

Summing this over $Q\subset B_R$ by using the multilinear estimate, \Cref{thm:mulc_simplified}, we get
\begin{align*}
	\| f \|_{L^2(Y)} \less E^{\frac{d-m}{m+1}}
	 \Big( |Y|  \max_{B_{R^{1/2}}} |Y\cap B_{R^{1/2}}| ^{m-1}  
	R^{-\frac{m(d+1)}{2}} \Big)^{\frac{1}{m(m+1)}}  \prod_{j=1}^m   \| f_{\tau_j} \|_{L^2}^{\frac{1}{m}}.
\end{align*} 
Using \eqref{eqn:lowerBoundQuantities} and $ \| f_{\tau_j} \|_{L^2} \les \| f\|_{L^2}$, the desired bound follows in the broad case.

\subsection{Narrow case}
In the narrow case, we proceed by induction on the scale $R$. Unlike in the broad case, this requires tracking powers of $K$ throughout the argument. In the final step, we use the induction hypothesis at the scale $R/K^2$, which results in the gain of a factor of $K^{-2\eps}$. Therefore, the accumulated factor $K^{O(\eps_1)}$ is harmless  provided that $\eps_1$ is chosen sufficiently small so that $O(\eps_1)\leq \eps$, which will be assumed from now on. Accordingly, throughout this section, we write $A \less B$ to mean $A \leq c_{\eps_1} K^{O(\eps_1)} B$ for some absolute constant $c_{\eps_1}>0$.

For the induction-on-scale argument, we may assume that  $R \gg_\eps 1$, so that $c_{\eps_1} \leq K^\eps$, and the desired estimate holds for all scales up to $R/2$; for the base case $R \ll_\eps 1$, the trivial bound $\| f\|_{L^2(Y)}\leq \| f\|_{L^2}$  and \eqref{eqn:lowerBoundQuantities} yield the claimed estimate. 

\subsubsection{Decoupling over narrow $K^2$-cubes}

We recall that \[ \| f\|_{L^2(B\cap Y)} \sim \| \sum_{\tau\in S(B)} f_\tau \|_{L^2(B\cap Y)}\] for each $B\in \Bn$. For each $\tau\in S(B)$, consider a wave packet decomposition
\[ f_\tau = \sum_{T_1 \in \T_{\tau}} f_{T_1},\]
where $\T_{\tau} \subset \T(K^2)$ and each $T_1\in \T_{\tau}$ has dimensions $K\times \cdots \times K \times K^2$ with the long side orthogonal to $\tau$. We recall that $2T_1$ denotes the concentric  $K^{2\eps_1}$-dilate of $T_1$. Since we are concerned with $f_\tau$ on $Y$, the contribution of $T_1\in \T_\tau$ such that $2T_1$ does not contain any unit cube from $Y$ is negligible. Hence, after trimming wave packets if necessary, we may assume that $2T_1$ contains at least one unit cube from $Y$, so  $K^{-(d+1)} \leq |2T_1\cap Y| / |T_1| \less 1$ for any $T_1\in \T_\tau$. 

By dyadic pigeonholing, there exists a dyadic $\eta$ with $K^{-(d+1)} \leq \eta \less 1$ and a sub-collection $\T_{\tau,\eta} \subset \T_\tau$ such that $|2T_1 \cap Y| \sim \eta |T_1|$ for every $T_1\in \T_{\tau,\eta}$ and 
\begin{align*}
	\|f\|_{L^2(Y)}^2 &\sim  \sum_{B\in \Bn} \| \sum_{\tau\in S(B)} f_\tau \|_{L^2(B\cap Y)}^2 \\
	&\les \log K \sum_{B\in \Bn}\| \sum_{\tau\in S(B)} \sum_{T_1 \in \T_{\tau,\eta}} f_{T_1} \|_{L^2(B\cap Y)}^2.
\end{align*}

Let $\T(K^2;B)$ denote the collection of $T_1\in \cup_{\tau\in S(B)} \T_{\tau,\eta}$ such that $2T_1$ intersects $B$. 
By the localization of $f_{T_1}$, the contribution of $T_1\notin \T(K^2;B)$ is negligible on $B$. Note that for each  $B\in \Bn$, $\sum_{T_1 \in \T(K^2;B)} f_{T_1}$ has $(O(K^{-1}),m-1)$-concentrated frequencies by the definition of narrow $K^2$-cubes and the fact that the angle between any two vectors in $G(\tau)$ is $O(K^{-1})$. By \Cref{thm:b}, with $m$ and $R$ replaced by $m-1$ and $K^2$, respectively, we have 
\[
\| \sum_{\tau\in S(B)} \sum_{T_1 \in \T_{\tau,\eta}} f_{T_1}\|_{L^2(B\cap Y)} \less \eta^{\frac{1}{m}}  \Big(    \sum_{T_1 \in \T(K^2;B)}  \|  f_{T_1} \|_{L^2}^2 \Big)^{\frac{1}{2}},
\]
which yields 
\begin{equation}\label{eqn:mainb}
	\| f\|_{L^2(Y)}^2 \less \eta^{\frac{2}{m}} \sum_{B\in \Bn}\sum_{T_1 \in \T(K^2;B)} \|  f_{T_1} \|_{L^2}^2. 
\end{equation}

\subsubsection{Organizing small tubes $T_1$}
For each $\tau$, we cover $B_R$ by finitely overlapping parallelepipeds $\sq$ of dimensions $K^{-1}R \times \cdots \times K^{-1}R  \times R$ with the long side perpendicular to $\tau$. Each $\sq$ is the $R/K^{2}$ dilation of some $T_1 \in \T_\tau$. We denote the collection of $\sq$ by $\bB_\tau$ and let $\bB = \cup_\tau \bB_\tau$.  In addition, given $\sq \in  \bB$, we denote by $\tau(\sq)$ the $\tau$ for which $\sq\in \bB_\tau$. 

For each $T_1 \in  \cup_{B\in \Bn} \T(K^2;B)$, we fix $\sq\in \bB$ such that $T_1 \in \T_{\tau(\sq)}$ and  $2T_1\subset \sq$, and denote it by $\sq(T_1)$.
For each $\sq\in \bB$, let $\T_\sq$ denote the collection of all $T_1 \in \cup_{B\in \Bn} \T(K^2;B)$ such that $\sq(T_1) = \sq$. Then 
\begin{align*}
	\sum_{B\in \Bn} \sum_{T_1 \in \T(K^2;B)} \|  f_{T_1} \|_{L^2}^2 
	\less \sum_{\sq} \sum_{T_1\in \T_\sq} \| f_{T_1} \|_{L^2}^2, 
\end{align*}
since $T_1 \in \T(K^2;B)$ for $\less 1$ narrow $K^2$-cubes $B\in \Bn$. 
Given the wave packet decomposition $f_\tau = \sum_{T_1\in \T_\tau} f_{T_1}$, we have
\[  \| f_{T_1} \|_{L^2}^2 \les_{\eps_1} \| f_{\tau} \|_{L^2(2T_1)}^2 + K^{-\frac{200d}{\eps_1}} \| f_{\tau} \|_{L^2}^2. \]
We let 
\[ Y_\sq := \cup_{T_1\in \T_\sq} 2T_1,\] 
so that
\begin{align*}
	\sum_{\sq} \sum_{T_1\in \T_\sq} \| f_{T_1} \|_{L^2}^2 &\less   \sum_{\sq} \sum_{T_1\in \T_\sq} \Big(\| f_{\tau(\sq)} \|_{L^2(2T_1)}^2 + K^{-\frac{200d}{\eps_1}} \| f_{\tau(\sq)} \|_{L^2}^2 \Big) \\
	&\less  \sum_{\sq}  \| f_{\tau(\sq)} \|_{L^2(Y_\sq)}^2 +R^{-100d} \| f \|_{L^2}^2,
\end{align*}
where we have used the fact that the dilated tubes $\{ 2T_1\}$ have $\less 1$ overlap. The term $R^{-100d} \| f \|_{L^2}$ is negligible and will therefore be omitted in what follows.

Let $\W_\sq$ denote the collection of $T\in \W_{\tau(\sq)}$ such that $2T$ intersects $\sq$ and define $f_\sq = \sum_{T\in \W_\sq} f_T$. Each $T$ can be contained in $\W_\sq$ for $O(1)$ boxes $\sq$. Then, on $Y_\sq$, we may replace $f_{\tau(\sq)} = \sum_{T\in \W_{\tau(\sq)}} f_T$ by $f_\sq$ up to a negligible error term. By combining \eqref{eqn:mainb} with preceding estimates, we obtain 
\begin{equation}\label{eqn:decK}
	\| f\|_{L^2(Y)}^2 \less \eta^{\frac{2}{m}}\sum_{\sq} \| f_{\sq} \|_{L^2(Y_\sq)}^2.
\end{equation}

We also note the following which will be used later.
\begin{equation}\label{eqn:L2ortho}
	\sum_{\sq} \|  f_\sq \|_{L^2}^2 \sim \sum_{\sq} \sum_{T\in \W_\sq} \|  f_T \|_{L^2}^2 \les \sum_{T\in \W} \| f_T \|_{L^2}^2 \sim \| f\|_{L^2}^2.
\end{equation}

\subsubsection{Parabolic scaling}
We fix $\sq \in \bB_\tau$. We handle $\|  f_{\sq} \|_{L^2(Y_{\sq})}^2$ by parabolic scaling followed by induction on scale. We briefly recall the standard parabolic scaling argument (see e.g. \cite{Guth1}) for completeness. 

Without loss of generality, we may assume that the surface $S\subset \R^d$ is the graph of a smooth function $h$ over $[-1,1]^{d-1}$ such that $h(0)=\partial h(0) = 0$ and $
\frac{1}{2} I \leq \partial^2 h \leq 2I$ on $[-1,1]^{d-1}$ as quadratic forms. 

We begin by giving a formal description of the covering of $S$ by caps $\tau$ of dimensions  $K^{-1}\times \cdots K^{-1} \times K^{-2}$. Partition $[-1,1]^{d-1}$ by cubes $\tau'$ of the form $c_{\tau'} + [-K^{-1},K^{-1}]^{d-1}$. Then $\tau$ can be identified with the image of $[-2,2]^d$ under the affine transformation $A_\tau  : \R^d \to \R^d$ defined  by
\[  A_{\tau} (\zeta, t)  = (c_{\tau'}+ K^{-1} \zeta, \, h(c_{\tau'})+ K^{-1} \partial h(c_{\tau'}) \cdot \zeta + K^{-2} t ),  \;\; (\zeta,t)\in \R^{d-1}\times \R.\] 
Note that 
\[ A_{\tau}^{-1} (\omega, s) = \Big( K(\omega-c_{\tau'}) , K^2 \big( s-h(c_{\tau'})-\partial h(c_{\tau'})\cdot(\omega - c_{\tau'}) \big) \Big). \]

Let $L_{\tau}$ denote the transpose of the linear part of $A_\tau$: 
\[ L_{\tau} (x',x_d):= (K^{-1} x'+K^{-1}x_d \partial h(c_\tau'), K^{-2}x_d ), \; (x',x_d) \in \R^{d-1}\times \R. \]
Let $R_1$ denote the new scale 
\[ R_1= R/K^2. \] 
We have defined $\sq$ and $T_1 \in \T_\sq$ so that they are translates of  $L_\tau^{-1}(B_{R_1}(0))$  and $L_\tau^{-1}(B_{1}(0))$, respectively.

For $\zeta \in \tau'$, define  
\[ \tilde{h}(\zeta) =  h(\zeta) - h(c_{\tau'}) - \partial h(c_{\tau'})\cdot (\zeta-c_{\tau'}). \]
Then 
\[ A_\tau^{-1}(\zeta, h(\zeta)) = (K(\zeta- c_{\tau'}), K^2\tilde{h}(\zeta) ). \]
Hence, the affine map $A_{\tau}^{-1}$ transforms $\{ (\zeta, h(\zeta)): \zeta \in \tau' \} \subset S$ to the surface $S_1\subset \R^d$ defined as the graph of the function $h_1$, where 
\[ h_1(\omega) =K^2  \tilde{h}(c_{\tau'} + K^{-1} \omega), \;\; \omega  \in [-1,1]^{d-1}. \]
The new surface $S_1$ has properties similar to $S$. Indeed, note that $h_1(0)=\partial h_1(0) = 0$, $\partial_{i} \partial_{j} h_1(\omega)=\partial_{i} \partial_{j}  h(c_{\tau'} + K^{-1} \omega )$, thus $\frac{1}{2} I\leq \partial^2 h_1 \leq 2I$ on $[-1,1]^{d-1}$ as quadratic forms. 

Each $\theta \subset \tau$ is transformed to $\theta_\tau := A_\tau^{-1} \theta = A_\tau^{-1} A_\theta ([-2,2]^d)$.  Let $c_{\theta_\tau'} = K(c_{\theta'}-c_{\tau'})$.
The composition $A_\tau^{-1} A_\theta$ is the affine transform 
\[  (A_\tau^{-1} A_\theta) (\zeta,t) := (c_{\theta_\tau'} + R_1^{-1/2} \zeta, \, h_1(c_{\theta_\tau'} ) + R_1^{-1/2} \partial h_1(c_{\theta_\tau'}) \cdot \zeta + R_1^{-1}t ), \]
which is associated with the canonical cap $\theta_\tau$ of dimensions $R_1^{-1/2}\times \cdots \times R_1^{-1/2} \times R_1^{-1}$ covering $N_{R_1^{-1}}(S_1)$.

Let $\hat{g_T} = \hat{f_T }\circ A_\tau$. Then $g_T$ is a wave packet associated with a tube in $\T(R_1)$.
Indeed, under the parabolic scaling, a wave packet $f_T$ localized to $T\in \W_\sq$ in physical space and $\theta \subset \tau$ in frequency space is transformed into the wave packet $g_T$ localized to $L_\tau(T)$ and $\theta_\tau$, respectively. Moreover, 
\[ g_\sq := \sum_{T\in \W_\sq} g_T\]
has $(O(E) R_1^{-1/2}, m)$-concentrated frequencies. Indeed, if $V$ is an $m$-dimensional subspace for which $\angle(v(T),V) \leq ER^{-1/2}$ for $T\in \W$, then $\angle(v(L_\tau(T)), L_\tau(V) ) \les KER^{-1/2}$ for $T\in \W_\sq$, which is a consequence of the following.  
\begin{lemma}
	Let $u, \, v \in \S^{d-1}$ be unit vectors and $L=L_\tau$. Then 
	\[ \angle(L  u, L  v) \les K \angle(u,v). \]
\end{lemma}
\begin{proof}
	We note that \[ \| L  u\| \geq \| L^{-1} \|^{-1}, \;\; \| L v\| \geq \| L^{-1} \|^{-1}. \] 
	Using the bound $\angle(x,y)  \sim \| \frac{x}{\|x\|} - \frac{y}{\|y\|} \| \les \frac{\| x-y\|}{\min(\|x\|, \|y\|)}$, we have
	\[ \angle(L  u, L  v) \les  \frac{\|L  u-Lv\|}{\min(\|L  u\|, \| Lv\|)}   \les \|L^{-1}\| \| L \| \| u - v\| \sim   \|L^{-1}\| \| L \| \angle( u, v).  \]
	The claim follows from the bounds $\| L  \| \les K^{-1}$ and $\| L^{-1} \| \les K^2$.
\end{proof}

\subsubsection{Induction on scale}\label{sec:induct}
Fix $\sq \in \bB_\tau$. 
Recall that $g_\sq$ is a sum of wave packets associated with tubes in $\T(R_1)$ and it has $(O(E)R_1^{-1/2},m)$-concentrated frequencies for $R_1 = R/K^2$. Moreover,  $L_\tau(Y_\sq)$ is a union of unit cubes contained in the cube $L_\tau(\sq)$ of side length $R_1$. By the induction hypothesis, we have
\begin{equation}\label{eqn:induction}
	 \frac{ \| f_\sq\|_{L^2(Y_\sq)}}{\| f_{\sq} \|_{L^2} } =\frac{ \| g_\sq\|_{L^2(L_{\tau}(Y_\sq))}}{\| g_{\sq} \|_{L^2} } \leq C_\eps R_1^\eps O(E)^\frac{d-m}{m+1} \mathcal{A}_{d,m,\alpha,\beta}(R_1,L_\tau(Y_\sq))^{\frac{1}{m(m+1)}}.  
\end{equation}

To bound the density parameter $\mathcal{A}_{d,m,\alpha,\beta}(R_1,L_\tau(Y_\sq))$, we use the following result.
\begin{lemma}
Let $1\leq \rho \leq R_1$. Then 
\[ 	\sup_{B_\rho} |L_\tau(Y_\sq) \cap B_\rho| \less \eta^{-1}K^{-d} \sup_{B_{K\rho}} |Y\cap B_{K\rho}|. \]
\end{lemma}
\begin{proof}
Note that $L_\tau^{-1}(B_\rho)$ is a parallelepiped  of dimensions $K\rho\times \cdots \times K\rho \times K^2 \rho$. Let $N$ denote the number of tubes $T_1\in \T_\sq$ such that $2T_1$ intersects $L_\tau^{-1}(B_\rho)$. Then we have
\begin{equation}\label{eqn:density}
	|L_\tau(Y_\sq) \cap B_\rho| =  |\det L_\tau| |Y_\sq \cap L_\tau^{-1}(B_\rho)| \leq  |\det L_\tau| |2T_1| N \less N
\end{equation}
since $|\det L_\tau| = K^{-(d+1)}$ and $|2T_1| \less K^{d+1}$.

If $2T_1$ intersects $L_\tau^{-1}(B_\rho)$, then it is contained in $2L_\tau^{-1}(B_\rho)$ which denotes the $K^{2\eps_1}$-dilate of $L_\tau^{-1}(B_\rho)$. Since the dilated tubes $\{2T_1\}$ have $\less 1$ overlap, $2 L_\tau^{-1}(B_\rho)$ contains at least $\gtrapprox N$ disjoint tubes $2T_1$ with $T_1\in \T_\sq$. Since $|Y\cap 2T_1| \sim \eta |T_1| \sim \eta K^{d+1}$ for $T_1\in \T_\sq$, we have 
\begin{equation}\label{eqn:N}
	\eta K^{d+1} N \less | Y\cap  2 L_\tau^{-1}(B_\rho) |. 
\end{equation}
On the other hand, by covering the parallelepiped $2 L_\tau^{-1}(B_\rho)$ by $\less K$ cubes of side-length $K\rho$, 
\[ | Y\cap  2 L_\tau^{-1}(B_\rho) |  \less  K \sup_{B_{K\rho}} |Y\cap B_{K\rho}|.\]
Combined with \eqref{eqn:density}  and \eqref{eqn:N}, this estimate yields
\[ \eta K^{d+1}	|L_\tau(Y_\sq) \cap B_\rho| \less  K \sup_{B_{K\rho}} |Y\cap B_{K\rho}|, \]
from which the claim follows.
\end{proof}

In particular, by the lemma, we have 
\[ 	\sup_{B_{R_1^{1/2}}} |L_\tau(Y_\sq) \cap B_{R_1^{1/2}}| \less 	\eta^{-1}  K^{-d} \sup_{B_{R^{1/2}}} |Y\cap B_{R^{1/2}}|.\]

Moreover, since $\rho \leq R_1^{1/2} \iff K\rho \leq R^{1/2}$, we have 
\begin{align*}
	\gamma_{\alpha, R_1^{1/2}} \big(L_\tau(Y_\sq) \big) &\less \eta^{-1} K^{\alpha-d}   \gamma_{\alpha,R^{1/2}}(Y) \\
	\overline{\gamma_{\beta, R_1}} \big(L_\tau(Y_\sq)\big) &\less \eta^{-1} K^{\beta-d} \overline{\gamma_{\beta, R}}( Y).
\end{align*}
Combining these bounds, we obtain
\[
	\mathcal{A}_{d,m,\alpha,\beta}(R_1, L_\tau(Y_\sq))  \less \eta^{-(m+1)} K^{-\beta+\alpha+m-d}  \mathcal{A}_{d,m,\alpha,\beta}(R, Y). 
\]
By the assumption $\beta \geq \alpha + m - d$, we have $K^{-\beta+\alpha+m-d}  \leq 1$. Thus, by   \eqref{eqn:induction}, we have 
\begin{align*}
	\frac{\| f_\sq \|_{L^2(Y_\sq)} }{\| f_\sq \|_{L^2}}
	&\less K^{-2\eps}  C_\eps R^\eps E^{\frac{d-m}{m+1}}  \eta^{-\frac1 m} \mathcal{A}_{d,m,\alpha,\beta}(R, Y)^{\frac{1}{m(m+1)}}.
\end{align*}
Recall that $\less K^{-2\eps}$ denotes $\leq c_{\eps_1} K^{O(\eps_1)} K^{-2\eps}$. Since  we  assume that $\eps_1$ is sufficiently small relative to $\eps$ and  $R\gtrsim_\eps 1$ is sufficiently large, we may assume that 
\[ c_{\eps_1} K^{O(\eps_1)} K^{-2\eps}  \leq  c_{\eps_1}	K^{-\eps} \leq 10^{-d}. \]
Therefore, summing the  square of the estimate on $\| f_\sq \|_{L^2(Y_\sq)}$ over $\sq$ using \eqref{eqn:decK} and \eqref{eqn:L2ortho} yields 
\begin{align*}
	\| f\|_{L^2(Y)} &\leq   C_\eps R^\eps E^{\frac{d-m}{m+1}}  \mathcal{A}_{d,m,\alpha,\beta}(R, Y)^{\frac{1}{m(m+1)}} 
	\| f\|_{L^2},
\end{align*}
which closes the induction.

\newcommand{\etalchar}[1]{$^{#1}$}

\end{document}